\documentclass{amsart}
\usepackage{graphicx}
\input xy
\xyoption{all}
\newtheorem{thm}{Theorem}

\theoremstyle{definition}

\theoremstyle{remark}

\theoremstyle{remark}
\newtheorem{rems}[thm]{Remarks}
\theoremstyle{definition}

\numberwithin{equation}{section}

\begin{document}

\title[tensor products of complete intersections]{On tensor products of complete intersections}%
\author{Javier Majadas}%
\address{Departamento de \'Algebra, Facultad de Matem\'aticas, Universidad de Santiago de Compostela, E15782 Santiago de Compostela, Spain}%
\email{j.majadas@usc.es}%


\begin{abstract}
  The study of regularity and complete intersection of a tensor product of commutative algebras possessing the same property started with Grothendieck in 1965 and has continued until today. Surprisingly, the homology theory of Andr\'e and Quillen, developed by these authors in 1967, has never been used for this study. With the help of this theory, we can (slightly) generalize the results known up to now. But even more important, we hope to convince the reader that this homology theory is the adequate tool to handle these problems: the proofs are very short and (assuming some flatness hypothesis) it allows to see clearly what extra hypotheses we need.
\end{abstract}
\maketitle

In this short paper, we are concerned with the problem of giving sufficient conditions to ensure that the tensor product of two commutative algebras possessing some property of "homological type" (regularity, complete intersection, ...) has the same property. Some results were already obtained by A. Grothendieck in 1965 (see e.g. \cite [6.7.1.1, 6.7.4.1]{EGA42}). In a different context, in 1967, M. Andr\'e \cite {AnLNM} and D. Quillen \cite {MIT} develop a homology theory of commutative algebras, giving in particular some criteria for regularity and complete intersection properties that are easy to handle. Moreover, this theory contains a result which allows to study the tensor product of two algebras \cite [19.5]{AnLNM}.

The study of this kind of properties of a tensor product has continued until today (see e.g. \cite {Wa}, \cite {Sh1}, \cite {BK1}, \cite {TY}, \cite {BK2}). Concerning the complete intersection property, the two strongest results included in the literature are: (i) if $A$ is a regular ring, $B$ a flat $A$-algebra which is a complete intersection ring and $C$ an $A$-algebra of finite type which is a regular ring, then $B\otimes_AC$ is complete intersection \cite [Part 1, Theorem 2]{Wa}; (ii) if $A$ is a field, $B$ and $C$ are $A$-algebras which are complete intersection as rings and such that $B\otimes_AC$ is a noetherian ring, then $B\otimes_AC$ is complete intersection \cite [Theorem 6]{TY}.

But none of these papers have used Andr\'e-Quillen homology. In this short note, we will see that the use of this theory makes things easier allowing in particular to obtain almost immediately a result stronger than those two just mentioned (Theorem \ref{ci}).

Something similar happens for the regularity property. Several results had been obtained, from the first \cite [6.7.4.1]{EGA42} until the most recent and strongest \cite [Theorem 2.11]{BK2}. Again, with the help of Andr\'e-Quillen homology, we obtain with ease a stronger result (Theorem \ref {regularity}).

A close inspection of our proofs reveals that in fact, once the results are stated, we could also find proofs of them without the need of Andr\'e-Quillen homology. But the advantage of using this theory is that we can formulate these problems as a simple vanishing question of a term in an exact sequence whose remaining terms are in close relation with flatness, regularity and complete intersection properties, so we can see at a first glance several conditions sufficient for the wanted result.\\

If $B$ is an (always commutative) $A$-algebra and $M$ a $B$-module, we will use Andr\'e-Quillen homology modules denoted $H_n(A,B,M)$ for all $n\geq 0$. We will refer to \cite {An1974} for the properties that we will use.

\begin{thm}\label{regularity}
Let $A$ be a noetherian ring, $B$ and $C$  two noetherian $A$-algebras such that $B\otimes_AC$ is a noetherian ring. For each maximal ideal $\mathfrak{q}$ of $B\otimes_AC$ let $\mathfrak{q}_A$, $\mathfrak{q}_B$, $\mathfrak{q}_C$ denote its contractions in $A$, $B$ and $C$ respectively and let $k(\mathfrak{q})$ be the residue field of $(B\otimes_AC)_{\mathfrak{q}}$. Assume that for each $\mathfrak{q}$ at least one of the three local $A_{\mathfrak{q}_A}$-algebras $B_{\mathfrak{q}_B}$, $C_{\mathfrak{q}_C}$ or $k(\mathfrak{q})$ is formally smooth for the topology of its maximal ideal \cite [19.3.1]{EGA41}. If $B$ and $C$ are regular rings, then $B\otimes_AC$ is regular.
\end{thm}

\begin{proof}
For each maximal ideal $\mathfrak{q}$ of $B\otimes_AC$, we have $Tor_i^{A_{\mathfrak{q}_A}}(B_{\mathfrak{q}_B},C_{\mathfrak{q}_C})=0$ for any $i>0$ (\cite [19.7.1]{EGA41} or \cite [16.18]{An1974}). We have an exact sequence \cite [5.21]{An1974}\\

\hspace{8mm} $ ... \rightarrow H_2(B_{\mathfrak{q}_B},k(\mathfrak{q}),k(\mathfrak{q}))\bigoplus H_2(C_{\mathfrak{q}_C},k(\mathfrak{q}),k(\mathfrak{q})) \rightarrow H_2((B\otimes_AC)_{\mathfrak{q}},k(\mathfrak{q}),k(\mathfrak{q}))$
$ \rightarrow H_1(A_{\mathfrak{q}_A},k(\mathfrak{q}),k(\mathfrak{q})) \xrightarrow{\alpha_{\mathfrak{q}}} H_1(B_{\mathfrak{q}_B},k(\mathfrak{q}),k(\mathfrak{q}))\bigoplus H_1(C_{\mathfrak{q}_C},k(\mathfrak{q}),k(\mathfrak{q})) \rightarrow ...$\\
\\
since $H_2((B\otimes_AC)_{\mathfrak{q}},k(\mathfrak{q}),k(\mathfrak{q})) = H_2(B_{\mathfrak{q}_B}\otimes_{A_{\mathfrak{q}_A}}C_{\mathfrak{q}_C},k(\mathfrak{q}),k(\mathfrak{q}))$ by \cite [5.27]{An1974}.

Since $B_{\mathfrak{q}_B}$ and $C_{\mathfrak{q}_C}$ are regular local rings, we have
$$H_2(B_{\mathfrak{q}_B},k(\mathfrak{q}),k(\mathfrak{q}))\bigoplus H_2(C_{\mathfrak{q}_C},k(\mathfrak{q}),k(\mathfrak{q}))=0$$
 \cite [6.26, 7.4]{An1974}. On the other hand, if $k(\mathfrak{q})$ is formally smooth over $A_{\mathfrak{q}_A}$, then $H_1(A_{\mathfrak{q}_A},k(\mathfrak{q}),k(\mathfrak{q}))=0$ (essentially \cite [16.17]{An1974}; details can be seen in \cite [2.3.5]{MR}). Thus $H_2((B\otimes_AC)_{\mathfrak{q}},k(\mathfrak{q}),k(\mathfrak{q}))=0$ and then $(B\otimes_AC)_{\mathfrak{q}}$ is regular. For the same argument, if $B_{\mathfrak{q}_B}$ or $C_{\mathfrak{q}_C}$ is formally smooth over $A_{\mathfrak{q}_A}$, from the Jacobi-Zariski exact sequences \cite [5.1]{An1974}
$$H_1(A_{\mathfrak{q}_A},B_{\mathfrak{q}_B},k(\mathfrak{q})) \rightarrow H_1(A_{\mathfrak{q}_A},k(\mathfrak{q}),k(\mathfrak{q})) \xrightarrow{\alpha_{\mathfrak{q}_B}} H_1(B_{\mathfrak{q}_B},k(\mathfrak{q}),k(\mathfrak{q}))$$
$$H_1(A_{\mathfrak{q}_A},C_{\mathfrak{q}_C},k(\mathfrak{q})) \rightarrow H_1(A_{\mathfrak{q}_A},k(\mathfrak{q}),k(\mathfrak{q})) \xrightarrow{\alpha_{\mathfrak{q}_C}} H_1(C_{\mathfrak{q}_C},k(\mathfrak{q}),k(\mathfrak{q}))$$
we see that $\alpha_{\mathfrak{q}_B}$ or $\alpha_{\mathfrak{q}_C}$ is injective, so that $\alpha_{\mathfrak{q}}$ is injective, and then again $H_2((B\otimes_AC)_{\mathfrak{q}},k(\mathfrak{q}),k(\mathfrak{q}))=0$.
\end{proof}

We will say that a noetherian local ring $R$ is complete intersection if its completion is a quotient of a regular local ring by an ideal generated by a regular sequence. We will say that a noetherian ring $R$ is complete intersection if all its local rings are complete intersection (if $R$ is a complete intersection local ring, all its localizations at primes are complete intersection local rings \cite {Av}).

\begin{thm}\label{ci}
Let $A$ be a ring, $B$ and $C$ two $A$-algebras such that for each maximal ideal $\mathfrak{q}$ of $B\otimes_AC$ at least one of the local $A_{\mathfrak{q}_A}$-algebras $B_{\mathfrak{q}_B}$, $C_{\mathfrak{q}_C}$ is flat (e.g. if $B$ is a flat $A$-algebra). Assume that $B\otimes_AC$ is a noetherian ring. If $B$ and $C$ are complete intersection then so is $B\otimes_AC$.
\end{thm}
\begin{proof}
As in the proof of theorem \ref{regularity} we have an exact sequence for each $\mathfrak{q}$\\

\hspace{8mm} $ ... \rightarrow H_3(B_{\mathfrak{q}_B},k(\mathfrak{q}),k(\mathfrak{q}))\bigoplus H_3(C_{\mathfrak{q}_C},k(\mathfrak{q}),k(\mathfrak{q})) \rightarrow H_3((B\otimes_AC)_{\mathfrak{q}},k(\mathfrak{q}),k(\mathfrak{q}))$
$ \rightarrow H_2(A_{\mathfrak{q}_A},k(\mathfrak{q}),k(\mathfrak{q})) \xrightarrow{\beta_{\mathfrak{q}}} H_2(B_{\mathfrak{q}_B},k(\mathfrak{q}),k(\mathfrak{q}))\bigoplus H_2(C_{\mathfrak{q}_C},k(\mathfrak{q}),k(\mathfrak{q})) \rightarrow ...$\\
\\

If $B_{\mathfrak{q}_B}$ is a flat $A_{\mathfrak{q}_A}$-algebra, then  the homomorphism $H_2(A_{\mathfrak{q}_A},k(\mathfrak{q}),k(\mathfrak{q})) \xrightarrow{\beta_{\mathfrak{q}_B}} H_2(B_{\mathfrak{q}_B},k(\mathfrak{q}),k(\mathfrak{q}))$ is flat by \cite {Av} (see \cite [4.2.2]{MR}), and similarly for $C$ instead of $B$. Thus $\beta_{\mathfrak{q}}$ is injective. On the other hand, using \cite [10.18]{An1974}, the proof of \cite [6.27]{An1974} shows that a local ring $(R,k)$ is complete intersection if and only if $H_3(R,k,k)=0$. So $H_3(B_{\mathfrak{q}_B},k(\mathfrak{q}),k(\mathfrak{q}))\bigoplus H_3(C_{\mathfrak{q}_C},k(\mathfrak{q}),k(\mathfrak{q}))=0$ by hypothesis and then $H_3((B\otimes_AC)_{\mathfrak{q}},k(\mathfrak{q}),k(\mathfrak{q}))=0$, showing that $(B\otimes_AC)_{\mathfrak{q}}$ is complete intersection.
\end{proof}

\begin{rems}
(i) In Theorem \ref{regularity}, if $A$ is quasi-excellent noetherian local ring and $B$ is a noetherian local formally smooth $A$-algebra, then for any prime ideal $\mathfrak{p}$ of $B$, $B_{\mathfrak{p}}$ is formally smooth over $A_{\mathfrak{p}_A}$ \cite {localisation}, so that the formal smoothness hypothesis on the theorem are satisfied.\\*
(ii) These results are valid even if $B\otimes_AC$ is not noetherian, provided we use the adequate definitions of regularity and complete intersection on this case \cite [6.10]{PSPM}, \cite {AnCI}. For instance, in the proof of Theorem \ref{ci}, we can use the exact sequence for all dimensions $n\geq 2$ instead of 2 and 3 only, and observe that since $B$ and $C$ are complete intersection, by the flatness assumptions we have that $A_{\mathfrak{q}_A}$ is complete intersection for any $\mathfrak{q}$ (\cite {Av}), and so $\beta_{\mathfrak{q}}$ vanishes in dimension 2 as in the proof, and in dimensions $\geq 3$ because $A_{\mathfrak{q}_A}$ is complete intersection. The proof of Theorem \ref{regularity} works with similar changes.\\*
(iii) Together with \cite [19.5]{AnLNM}, another result that seemed to have gone unnoticed for some time concerning the tensor product of two algebras is the formula dim$(E\otimes_kL)$ = min\{tr.deg.$E|k$,tr.deg.$L|k$\}, for two field extensions. This formula was proved in \cite [Errata et Addenda, Err$_{IV}$,19, page 349]{EGA44} but remained apparently unnoticed until it was obtained again ten years later in \cite {Sh2}.\\*
\end{rems}


\end{document}